\documentclass[11pt,twoside]{article}
\usepackage{mathrsfs}
\usepackage{amsmath}
\usepackage{amsthm}
\usepackage{amsfonts}
\usepackage{amssymb}
\usepackage{color}
\usepackage{latexsym}
\usepackage[all]{xy}

\date{\empty}
\allowdisplaybreaks[4] 

\numberwithin{equation}{section} \theoremstyle{plain}
\newtheorem*{thm*}{Main Theorem}
\newtheorem{theorem}{Theorem}[section]
\newtheorem{corollary}[theorem]{Corollary}
\newtheorem*{corollary*}{Corollary}

\newtheorem*{claim*}{Claim}
\newtheorem{lemma}[theorem]{Lemma}
\newtheorem*{lemma*}{Lemma}

\newtheorem*{proposition*}{Proposition}
\newtheorem{remark}[theorem]{Remark}
\newtheorem*{remark*}{Remark}

\newtheorem*{example*}{Example}

\newtheorem*{question*}{Question}

\newtheorem*{definition*}{Definition}

\newtheorem*{acknowledgements*}{ACKNOWLEDGEMENTS}

\newcommand{\core}[1]{#1^{\tiny{\textcircled{\tiny \#}}}}

\newcommand{\bc}[1]{#1^{(b,c)}}
\newcommand{\uv}[1]{#1^{(u,v)}}
\newcommand{\bcw}[1]{#1^{(eb,cf)}}
\newcommand{\along}[1]{#1^{\|d}}

\begin{document}
\begin{center}
{\large \bf Existence criteria and expressions of the $(b,c)$-inverse in rings and its applications}

\vspace{0.4cm} {\small \bf Sanzhang Xu }
\footnote{ Sanzhang Xu (Corresponding author E-mail: xusanzhang5222@126.com): School of Mathematics, Southeast University, Nanjing 210096, China.}
\vspace{0.4cm} {\small \bf and Julio Ben\'{\i}tez}
\footnote{Julio Ben\'{\i}tez (E-mail: jbenitez@mat.upv.es): Universitat Polit\`{e}cnica de Val\`{e}ncia, Instituto de Matem\'{a}tica Multidisciplinar, Valencia, 46022, Spain.}

\end{center}
\bigskip

{ \bf  Abstract:}  \leftskip0truemm\rightskip0truemm
Let $R$ be a ring.
Existence criteria for the $(b,c)$-inverse are given.
We present explicit expressions for the $(b,c)$-inverse by using inner inverses.
We answer the question when the $(b,c)$-inverse of $a\in R$ is an inner inverse of $a$.
As applications, we give a unified theory of some well-known results of the $\{1,3\}$-inverse, the $\{1,4\}$-inverse,
the Moore-Penrose inverse, the group inverse and the core inverse.


 \bigskip

\section { \bf Introduction}

Throughout this paper, $R$ denotes a unital ring.
In \cite[Definition 1.3]{Dc}, Drazin introduced a new class of outer inverses in the setting of semigroups, namely, the $(b,c)$-inverse.
Let $a,b,c\in R$. We say that $y\in R$ is the {\em $(b,c)$-inverse} of $a$ if we have
\begin{equation} \label{abcira}
y\in (bRy)\cap (yRc),~yab=b~\text{and}~cay=c.
\end{equation}
If such $y\in R$ exists, then it is unique and denoted by $\bc{a}$.
The $(b,c)$-inverse is a generalization of the Moore-Penrose inverse, the Drazin
inverse, the group inverse and the core inverse.
Many existence criteria and properties of the $(b,c)$-inverse can be found in \cite{BBJ,BK,Dc,Df,KCC,KVC,Rg,RMa,Wh,WCC} etc.
In \cite[Definition 6.2 and 6.3]{Dc}, Drazin introduced the hybrid $(b,c)$-inverse and the annihilator $(b,c)$-inverse of $a$.
We call that $y\in R$ is the {\em hybrid $(b,c)$-inverse} of $a$ if we have $yay=y$, $yR=bR$ and $y^{\circ}=c^{\circ}$.
We call that $y\in R$ is the {\em annihilator $(b,c)$-inverse} of $a$ if we have $yay=y$, ${}^{\circ}y={}^{\circ}b$ and $y^{\circ}=c^{\circ}$.
By \cite[Theorem 6.4]{Dc}, if the the hybrid $(b,c)$-inverse (resp. the annihilator $(b,c)$-inverse) of $a$ exists, then it is unique.

In \cite{Mb}, Mary introduced a new type of generalized inverse, namely, the inverse along an element.
This inverse depends on Green's relations \cite{G}.
Let $a,d\in R$. We say that $a$ is {\em invertible along $d$} if there exists $y\in R$ such that
\begin{equation} \label{iiranlongia}
yad=d=day,~yR\subseteq dR~\text{and}~Ry\subseteq Rd.
\end{equation}
If such $y$ exists, then it is unique and denoted by $a^{\| d}$.
The inverse along an element extends some known generalized inverses, for example, the group inverse, the Drazin inverse
and the Moore-Penrose inverse. Many existence criteria of the inverse along an element can be found in \cite{Mb,MPc} etc.
By the definition of the inverse along $d$, we have that $a^{\| d}$ is the $(d,d)$-inverse of $a$.

The following notations $aR = \{ax \mid x\in R\}$, $Ra = \{xa \mid x\in R\}$,
$^{\circ}a = \{x\in R \mid xa = 0\}$ and $a^{\circ} = \{x \in R \mid ax = 0\}$  will be used in the sequel.
An involutory ring $R$ means that $R$ is a unital ring with involution, i.e., a ring with unity $1$, and a mapping $a\mapsto a^*$
in $R$ that satisfies $(a^*)^*=a$, $(ab)^*=b^*a^*$ and $(a+b)^*=a^*+b^*$, for all $a,b\in R$.
For convenience, we will also assume that $R$ has an involution.
The notations of the $\{1,3\}$-inverse, the $\{1,4\}$-inverse, the Moore-Penrose inverse, the group inverse and the core inverse etc. can be found in \cite{XCZ}.
\section { \bf Preliminaries}
In this section, we will collect and present some useful preliminaries, which will be used in the sequel.
\begin{lemma}   \label{134invese}
Let $a\in R$. Then
\begin{itemize}
\item[{\rm (1)}] \emph{\cite[p.201]{Ha}} $a\in R^{\{1,3\}}$ with $x\in a\{1,3\}$ if and only if $x^{\ast}a^{\ast}a=a$;
\item[{\rm (2)}] \emph{\cite[p.201]{Ha}} $a\in R^{\{1,4\}}$ with $y\in a\{1,4\}$ if and only if $aa^{\ast}y^{\ast}=a$;
\item[{\rm (3)}] \emph{\cite[Theorem 2]{HC}} $a\in R^{\{1,3\}}$ if and only if $R=Ra^{\ast} \oplus {}^{\circ}a$;
\item[{\rm (4)}] \emph{\cite[Theorem 3]{HC}} $a\in R^{\{1,4\}}$ if and only if $R=a^{\ast}R \oplus a^{\circ}$.
\end{itemize}
\end{lemma}

\begin{lemma}  \label{pirgroupa}
Let $a\in R$. Then
\begin{itemize}
\item[{\rm (1)}] \emph{\cite[Proposition 7]{Ha}} $a\in R^\#$ if and only if $R=aR\oplus a^{\circ}$. 
\item[{\rm (2)}] \emph{\cite[Proposition 7]{Ha}} $a\in R^\#$ if and only if $R=Ra\oplus {}^{\circ}a$. 
\item[{\rm (3)}] \emph{\cite[Proposition 8.22]{BRB}} $a\in R^\#$ if and only if $a^2x=a$ and $ya^2=a$ both have solutions. 
\end{itemize}
\end{lemma}

\begin{lemma}\emph{\cite[Lemma 8]{N}} \label{annihilator}
Let $a, b\in R$. Then:
\begin{itemize}
\item[{\rm (1)}] $aR\subseteq bR$ implies ${}^{\circ}b\subseteq {}^{\circ}a$ and the converse
is valid whenever $b$ is regular;
\item[{\rm (2)}] $Ra\subseteq Rb$ implies $b^{\circ}\subseteq a^{\circ}$ and the converse
is valid whenever $b$ is regular.
\end{itemize}
\end{lemma}

\begin{lemma}\emph{\cite[Lemma 3.2]{WCC}} \label{pirregulara}
Let $a, b\in R$. Then:
\begin{itemize}
\item[{\rm (1)}] Let $aR=bR$. If $a$ is regular, then $b$ is regular;
\item[{\rm (2)}] Let $Ra=Rb$. If $a$ is regular, then $b$ is regular.
\end{itemize}
\end{lemma}

\begin{lemma} \emph{\cite[lemma 3.1]{WCC}} \label{pirdcca}
Let $a, y\in R$ with $yay=y$. Then:
\begin{itemize}
\item[{\rm (1)}] $yaR=yR$;
\item[{\rm (2)}] $Ray=Ry$;
\item[{\rm (3)}] $a^{\circ}\cap yR=a^{\circ}\cap yaR=\{0\}$;
\item[{\rm (4)}] ${}^{\circ}a\cap Ry={}^{\circ}a\cap Ray=\{0\}$.
\end{itemize}
\end{lemma}

\begin{lemma} \emph{\cite[Theorem 7]{Mb}} \label{alongd}
Let $a,d\in S$. Then the following statements are equivalent:
\begin{itemize}
\item[{\rm (1)}] $a^{\| d}$ exists;
\item[{\rm (2)}] $dR\subseteq daR$ and $(da)^\#$ exists;
\item[{\rm (3)}] $Rd\subseteq Rad$ and $(ad)^\#$ exists.
\end{itemize}
In this case,
$$a^{\| d}=d(ad)^\#=(da)^\#d.$$
\end{lemma}


\begin{lemma} \emph{\cite[Theorem 2.2]{MPc}} \label{alongo}
Let $a,d\in S$. Then $a$ is invertible along $d$ if and only if $dR=dadR$ and $Rd=Rdad$.
\end{lemma}

\begin{lemma} \emph{\cite[Theorem 2.1 (ii) and Proposition 6.1]{Dc}} \label{abcire}
Let $a,b,c\in R$. Then $y\in R$ is the $(b,c)$-inverse of $a$ if and only if $yay=y$, $yR=bR$ and $Ry=Rc$.
\end{lemma}

\begin{lemma} \emph{\cite[Theorem 2.2]{Dc}}  \label{abcirg}
Let $a,b,c\in R$. Then there exists at least one $(b,c)$-inverse of $a$ if and only if $b\in Rcab$ and $c\in cabR$.
\end{lemma}

\begin{lemma} \emph{\cite[Proposition 3.3]{WCC}} \label{abcirf}
Let $a,b,c\in R$. If there exists a $(b,c)$-inverse of $a$, then $cab$, $b$ and $c$ are regular.
\end{lemma}


\begin{lemma} \emph{\cite[Proposition 2.7]{KCC}} \label{abcirl}
Let $a,b,c\in R$. Then the following are equivalent:
\begin{itemize}
\item[{\rm (1)}] $a$ is $(b,c)$-invertible;
\item[{\rm (2)}] $c\in R^-$, $a^{\circ} \cap bR=\{0\}$ and $R=abR\oplus c^{\circ}$;
\item[{\rm (3)}] $b\in R^-$, ${}^{\circ}a\cap Rc=\{0\}$ and $R=Rca\oplus {}^{\circ}b$.
\end{itemize}
\end{lemma}

By \cite[Theorem 2.9]{KCC} and the definitions of hybrid $(b,c)$-inverse and annihilator $(b,c)$-inverse, we have the following lemma.
\begin{lemma} \label{abcirn}
Let $a,b,c,y\in R$. Then the following are equivalent:
\begin{itemize}
\item[{\rm (1)}] $y$ is the $(b,c)$-inverse of $a$;
\item[{\rm (2)}] $c\in R^-$, $y$ is the hybrid $(b,c)$-inverse of $a$;
\item[{\rm (3)}] $b,c\in R^-$, $y$ is the annihilator $(b,c)$-inverse of $a$.
\end{itemize}
\end{lemma}

In \cite[Theorem 2.11]{KCC}, the authors gave a generalization of \cite[Theorem 2.1]{Wh}.
By Lemma~\ref{pirgroupa} and \cite[Theorem 2.11]{KCC}, we have the following lemma.
\begin{lemma} \label{ats2ringgroupa}
Let $a,b,c,d\in R$, $dR=bR$ and $d^{\circ}=c^{\circ}$. If $a$ is $(b,c)$-invertible, then
$ad$ and $da$ are group invertible. Furthermore, we have
\begin{equation} \label{ats2ringgroupc}
\bc{a}=\along{a}=d(ad)^\#=(da)^\#d.
\end{equation}
\end{lemma}
\begin{proof}
Since $a$ is $(b,c)$-invertible, then $b$ and $c$ are regular by Lemma~\ref{abcirf}.
By $dR=bR$ and Lemma~\ref{pirregulara}, we have $d$ is regular, thus
$d^{\circ}=c^{\circ}$ if and only if $Rd=Rc$.
Let $x$ be the $(b,c)$-inverse of $a$. Then $xax=x$, $xR=bR$ and $Rx=Rc$.
Thus $xax=x$, $xR=dR$ and $Rx=Rd$, which implies that $a$ is invertible along $d$ by \cite[Lemma 3]{Mb}.
Therefore, the proof is finished by Lemma~\ref{alongd}.
\end{proof}

\section { \bf Existence criteria of the $(b,c)$-inverses and its applications}

In this section, necessary and sufficient conditions of the $(b,c)$-invertibility are given and we present explicit expressions for the $(b,c)$-inverse by using inner inverses.
In Theorem~\ref{bcuva}, we will give a generalization of the well-known results in \cite[Theorem 2.1]{Wh}.
We answer the question when the $(b,c)$-inverse of $a$ is an inner inverse of $a$.
In Theorem~\ref{inofbca} and Theorem~\ref{inofbcdua}, we will give a unified theory of some well-known results
of the $\{1,3\}$-inverse, the $\{1,4\}$-inverse, the Moore-Penrose inverse, the group inverse and the core inverse.

\begin{theorem} \label{anihilata}
Let $a,b,c\in R$. Then the following are equivalent:
\begin{itemize}
\item[{\rm (1)}] $a$ is $(b,c)$-invertible;
\item[{\rm (2)}] $c\in R^-$, $(ab)^{\circ}=b^{\circ}$ and $R=abR\oplus c^{\circ}$;
\item[{\rm (3)}] $b\in R^-$, ${}^{\circ}(ca)={}^{\circ}c$ and $R=Rca\oplus {}^{\circ}b$.
\end{itemize}
\end{theorem}
\begin{proof}
$(1)\Rightarrow (2)$. By Lemma~\ref{abcirl}, we have $c\in R^-$ and $R=abR\oplus c^{\circ}$.
Let $y$ be the $(b,c)$-inverse of $a$, then $b=yab$.
For arbitrary $u\in (ab)^{\circ}$, we have $bu=yabu=0$, which implies that $(ab)^{\circ}\subseteq b^{\circ}$.
Thus $(ab)^{\circ}=b^{\circ}$ because $b^{\circ}\subseteq (ab)^{\circ}$ is trivial.

$(2)\Rightarrow (1)$.
Let  $v\in a^{\circ} \cap bR$.
Then $av=0$ and $v=br$ for some $r\in R$.
Thus $abr=av=0$, that is $r\in (ab)^{\circ}$.
The condition $(ab)^{\circ}=b^{\circ}$ gives that $r\in b^{\circ}$,
then $v=br=0$.
Therefore, $a$ is $(b,c)$-invertible by Lemma~\ref{abcirl}.

The proof of $(1)\Leftrightarrow(3)$ is similar to the proof of $(1)\Leftrightarrow(2)$.
\end{proof}

Let $a\in R$. By \cite[p.1910]{Dc}, we have that
$a$ is Moore-Penrose invertible if and only if $a$ is $(a^\ast,a^\ast)$-invertible,
$a$ is Drazin invertible if and only if exists $k\in \mathbb{N}$ such that $a$ is $(a^k,a^k)$-invertible and
$a$ is group invertible if and only if $a$ is $(a,a)$-invertible,
By \cite[Theorem 4.4]{RDDb}, we have the $(a,a^\ast)$-inverse coincides with the core inverse of $a$ and
the $(a^\ast,a)$-inverse coincides with the dual core inverse of $a$.
Thus, by Theorem~\ref{anihilata}, we can get corresponding results of the Moore-Penrose inverse, Drazin inverse,
core inverse and dual core inverse. Leaving the deeper details to the reader to research.

The following three lemmas will be useful in the sequel.
\begin{lemma} \label{ats2innera}
Let $a,b,c\in R$ such that $cab$ is regular. Let $x=b(cab)^-c$ for arbitrary $(cab)^-\in (cab)\{1\}$.
Then the following are equivalent:
\begin{itemize}
\item[{\rm (1)}] $xax=x$ and $bR=xR$;
\item[{\rm (2)}] $xax=x$ and $bR\subseteq xR$;
\item[{\rm (3)}] $Rb=Rcab$;
\item[{\rm (4)}] $b\in R^-$ and $b^{\circ}=(cab)^{\circ}$.
\end{itemize}
\end{lemma}
\begin{proof}
$(1)\Rightarrow (2)$ is trivial.

$(2)\Rightarrow (3)$. Suppose that $xax=x$ and $bR\subseteq xR$. Then $b=xab=b(cab)^-cab\in Rcab$,
thus $Rb=Rcab$.

$(3)\Rightarrow (1)$. Since $Rb=Rcab$ and $cab$ is regular, then
\begin{equation} \label{ats2innerbb}
b=b(cab)^-cab=[b(cab)^-c]ab=xab.
\end{equation}
By $(\ref{ats2innerbb})$ and $x=b(cab)^-c$, we have
\begin{equation*}
xax=xab(cab)^-c=b(cab)^-c=x~,xR\subseteq bR~\text{and}~bR\subseteq xR.
\end{equation*}
Thus, $xax=x$ and $bR=xR$.

$(3)\Leftrightarrow (4)$. Since $(1)\Leftrightarrow (3)$ and $x$ is regular, we have that $b$ is regular by
Lemma~\ref{pirregulara}. Thus $(3)\Leftrightarrow (4)$ by Lemma~\ref{annihilator} and the regularity of $cab$.
\end{proof}

The following lemma is the corresponding result of Lemma~\ref{ats2innera}.
\begin{lemma} \label{ats2innerabdd}
Let $a,b,c\in R$ such that $cab$ is regular. Let $x=b(cab)^-c$ for arbitrary $(cab)^-\in (cab)\{1\}$.
Then the following are equivalent:
\begin{itemize}
\item[{\rm (1)}] $xax=x$ and $Rx=Rc$;
\item[{\rm (2)}] $xax=x$ and $Rc\subseteq Rx$;
\item[{\rm (3)}] $cR=cabR$;
\item[{\rm (4)}] $c\in R^-$ and ${}^{\circ}c={}^{\circ}(cab)$.
\end{itemize}
\end{lemma}

\begin{lemma} \label{ats2iannnera}
Let $a,b,c\in R$ such that $cab$ is regular. Let $x=b(cab)^-c$ for arbitrary $(cab)^-\in (cab)\{1\}$.
Then the following are equivalent:
\begin{itemize}
\item[{\rm (1)}] $xax=x$ and $x^{\circ}=c^{\circ}$;
\item[{\rm (2)}] $xax=x$ and $x^{\circ}\subseteq c^{\circ}$;
\item[{\rm (3)}] $cR=cabR$.
\end{itemize}
\end{lemma}
\begin{proof}
$(1)\Rightarrow (2)$ is trivial.

$(2)\Rightarrow (3)$. Suppose that $xax=x$ and $x^{\circ}\subseteq c^{\circ}$. Then $c=cax=cab(cab)^-c\in cabR$,
thus $cR=cabR$.

$(3)\Rightarrow (1)$. Since $cR=cabR$ and $cab$ is regular, then
\begin{equation} \label{ats2innerb}
c=cab(cab)^-c=ca[b(cab)^-c]=cax.
\end{equation}
By $(\ref{ats2innerb})$ and $x=b(cab)^-c$, we have
\begin{equation*}
xax=b(cab)^-cax=b(cab)^-c=x.
\end{equation*}
Let $u\in c^{\circ}$. Then $xu=b(cab)^-cu=0$, that is $c^{\circ}\subseteq x^{\circ}$.
Let $v\in x^{\circ}$. Then $cv=caxv=0$ by $(\ref{ats2innerb})$, that is $x^{\circ}\subseteq c^{\circ}$.
Thus, $xax=x$ and $x^{\circ}=c^{\circ}$.
\end{proof}

Thus, by Lemma~\ref{abcirn}, Lemma~\ref{ats2innera} and Lemma~\ref{ats2iannnera}, we have the following theorem,
in which we give an explicit expression for the $(b,c)$-inverse,
which reduces to the inner inverse.
\begin{theorem} \label{informuast2a}
Let $a,b,c\in R$ such that $cab$ is regular. Let $x=b(cab)^-c$ for arbitrary $(cab)^-\in (cab)\{1\}$.
Then the following are equivalent:
\begin{itemize}
\item[{\rm (1)}] $x$ is the $(b,c)$-inverse of $a$;
\item[{\rm (2)}] $xax=x$, $bR\subseteq xR$ and $x^{\circ}\subseteq c^{\circ}$;
\item[{\rm (3)}] $b^{\circ}=(cab)^{\circ}$ and $cR=cabR$.
\end{itemize}
\end{theorem}

By the fact that if $a\in R$ is invertible along $d$ if and only if $a$ is $(d,d)$-invertible, we have the following corollary.
\begin{corollary} \label{informuast2aal}
Let $a,d\in R$ such that $dad$ is regular. Let $x=d(dad)^-d$ for arbitrary $(dad)^-\in (dad)\{1\}$.
Then the following are equivalent:
\begin{itemize}
\item[{\rm (1)}] $x$ is the inverse along $d$ of $a$;
\item[{\rm (2)}] $xax=x$, $dR\subseteq xR$ and $x^{\circ}\subseteq d^{\circ}$;
\item[{\rm (3)}] $d^{\circ}=(dad)^{\circ}$ and $dR=dadR$.
\end{itemize}
\end{corollary}

In \cite[Definition 1.2]{Df} and \cite[Definition 2.1]{KVC}, the authors introduced the one-sided $(b,c)$-inverse in rings.
Let $a,b,c\in R$. We call that $x\in R$ is a {\em left $(b,c)$-inverse} of $a$ if we have
\begin{equation} \label{leftbcda}
Rx\subseteq Rc~\text{and}~xab=b.
\end{equation}
We call that $y\in R$ is a {\em right $(b,c)$-inverse} of $a$ if we have
\begin{equation} \label{leftbcda}
yR\subseteq bR~\text{and}~cay=c.
\end{equation}
\begin{lemma} \emph{\cite[Proposition 2.8]{KVC}} \label{informuast2aarfsda}
Let $a,b,c\in R$. Then $y$ is a left $(b,c)$-inverse of $a$ if and only if $y^\ast$ is a right
$(c^\ast,b^\ast)$-inverse of $a^\ast$. 
\end{lemma}

\begin{theorem}
Let $a,b,c\in R$ such that $cab$ is regular. Then
\begin{itemize}
\item[{\rm (1)}]
if $a$ is both left $(b,c)$-invertible, then
a general solution of the left $(b,c)$-inverse of $a$ is
\begin{equation*}
b(cab)^-c+v[1-cab(cab)^-]c,
\end{equation*}
where $v\in R$ is arbitrary;
\item[{\rm (2)}]
if $a$ is both left and right $(b,c)$-invertible, then
a general solution of the right $(b,c)$-inverse of $a$ is
\begin{equation*}
b(cab)^-c+b[1-(cab)^-cab]u,
\end{equation*}
where $u\in R$ is arbitrary.
\end{itemize}
\end{theorem}
\begin{proof}
$(1)$. Let $x$ be a left $(b,c)$-inverse of $a$. Then we have $xab=b$ and $Rx\subseteq Rc$.
Thus $x=sc$ for some $s\in R$ and $b=xab=scab$. 
A general solution of $b=scab$ is 
\begin{equation*} 
b(cab)^-+v[1-cab(cab)^-],
\end{equation*}
where $v\in R$ is arbitrary.
Let $y=b(cab)^-c+v[1-cab(cab)^-]c$. Next we will check $y$ is a left $(b,c)$-inverse of $a$.
\begin{equation} \label{leftbsgbda}
yab=b(cab)^-cab+v[1-cab(cab)^-]cab=b(cab)^-cab.
\end{equation}
By Lemma~\ref{ats2innera} and \cite[Theorem 2.6]{KVC}, we have $Rb=Rcab$.
Since $cab$ is regular, then the condition $Rb=Rcab$ implies
$b=rcab=rcab(cab)^-cab=b(cab)^-cab$, thus $b=yab$ by $(\ref{leftbsgbda})$.
Thus $y$ is a left $(b,c)$-inverse of $a$ by $Ry\subseteq Rc$ is trivial.

$(2)$ follows from $(1)$ and Lemma~\ref{informuast2aarfsda}.
\end{proof}

\begin{theorem}
Let $a,b,c\in R$. If $a$ is both left and right $(b,c)$-invertible, then the left inverse of $a$ and the right inverse of $a$ are unique.
Moreover, the left $(b,c)$-inverse of $a$ coincides with the right $(b,c)$-inverse.
\end{theorem}
\begin{proof}
Let $x$ be a left $(b,c)$-inverse of $a$ and $y_1$ be a right $(b,c)$-inverse of $a$. Then we have
$Rx\subseteq Rc$, $xab=b$, $y_1R\subseteq bR$ and $cay_1=c$. Thus
$x=rc$ and $y_1=bs$ for some $r,s\in R$. Therefore,
\begin{equation*}
\begin{split}
x&=rc=rcay_1=xay_1;\\
y_1&=bs=xabs=xay_1.
\end{split}
\end{equation*}
That is $x=y_1$. If $y_2$ is a another right $(b,c)$-inverse of $a$, in a similar manner, we have $x=y_2$.
Then $y_1=y_2$ by $x=y_1$ and $x=y_2$, that is the right $(b,c)$-inverse of $a$ is unique.
In a similar way, the left $(b,c)$-inverse is unique, and by the previous reasoning, these two inverses are equal.
\end{proof}


\begin{theorem}
Let $a,b,c\in R$. Then the following are equivalent:
\begin{itemize}
\item[{\rm (1)}] $a$ is $(b,c)$-invertible;
\item[{\rm (2)}] $b,c\in R^-$, ${}^{\circ}c={}^{\circ}(cab)$, $R=Rc\oplus {}^{\circ}(ab)$ and $Rb=Rab$;
\item[{\rm (3)}] $b,c\in R^-$, $b^{\circ}=(cab)^{\circ}$, $R=bR\oplus (ca)^{\circ}$ and $cR=caR$;
\item[{\rm (4)}] $cab\in R^-$, $Rc\subseteq Rb(cab)^-c$, $R=Rc\oplus {}^{\circ}(ab)$ and $Rb=Rab$;
\item[{\rm (5)}] $cab\in R^-$, $bR\subseteq b(cab)^-cR$, $R=bR\oplus (ca)^{\circ}$ and $cR=caR$.
\end{itemize}
\end{theorem}
\begin{proof}
$(1)\Rightarrow (2)$. Suppose that $y$ is the $(b,c)$-inverse of $a$. Then the condition $yab=b$ implies that
that $Rb=Rab$. By Lemma~\ref{annihilator}, Lemma~\ref{abcirg} and Lemma~\ref{abcirf}, we have $b,c,cab \in R^-$
and ${}^{\circ}c={}^{\circ}(cab)$.
Since $1=ay+(1-ay)$, $ay\in Ry=Rc$ by Lemma~\ref{abcire} and $(1-ay)ab=ab-ayab=ab-ab=0$, thus $R= Rc+{}^{\circ}(ab)$.
Let $w\in Rc\cap {}^{\circ}(ab)$. Then $w=tc$ and $wab=0$ for some $t\in R$.
Thus $tcab=wab=0$, that is $t\in {}^{\circ}(cab)$.
By ${}^{\circ}c={}^{\circ}(cab)$, we have $tc=0$, i.e. $w=0$.
Therefore, $R=Rc\oplus {}^{\circ}(ab)$.

$(2)\Rightarrow (1)$. The condition $R=Rc\oplus {}^{\circ}(ab)$ implies that $1=uc+v$, where $u\in R$
and $v\in {}^{\circ}(ab)$. Then $ab=ucab+vab=ucab\in Rcab$, then $Rb=Rcab$ by $Rb=Rab$.
Since $b\in R^-$ and $Rb=Rcab$, then $cab$ is regular by Lemma~\ref{pirregulara}.
Let $x=b(cab)^-c$.  By Theorem~\ref{informuast2a}, we have $x$ is the $(b,c)$-inverse of $a$.

$(2)\Rightarrow (4)$. Since $(1)\Leftrightarrow(2)$, it is easy to check $(4)$ by the proof of $(2)\Rightarrow (1)$.

$(4)\Rightarrow (2)$. By the proof of $(2)\Rightarrow (1)$, we have $Rb=Rcab$, then $b$ is regular by $cab\in R^-$
and Lemma~\ref{pirregulara}. Let $x=b(cab)^-c$. Since $Rb=Rcab$, we have $xax=x$ by Lemma~\ref{ats2innera}.
The condition $Rc\subseteq Rb(cab)^-c$ implies that $Rc=Rx$, thus $c$ is regular by $xax=x$.
The proof is finished by Theorem~\ref{informuast2a}.

The proofs of $(1)\Leftrightarrow(3)$ and $(3)\Leftrightarrow(5)$ are similar to the proofs of
$(1)\Leftrightarrow(2)$ and $(2)\Leftrightarrow(4)$, respectively.
\end{proof}


In \cite[Theorem 2.11]{KCC}, the authors gave a generalization of \cite[Theorem 2.1]{Wh}.
In the following theorem, we will present a generalization of \cite[Theorem 2.11]{KCC}, which reduces to the $(b,c)$-inverse of $a$.
As applications of the Lemma~\ref{bcuva}, we have that \cite[Theorem 2.14]{KCC} and \cite[Theorem 2.13]{KCC}
can be generalized.

\begin{lemma} \emph{\cite[Remark 2.2 (i)]{BK}} \label{bcuva}
Let $a,d,u,v\in R$. If $bR=uR$ and $Rc=Rv$, then $a$ is $(b,c)$-invertible if and only if $a$ is $(u,v)$-invertible.
In this case, we have $\bc{a}=\uv{a}$.
\end{lemma}
\begin{proof}
Let $bR=uR$ and $Rc=Rv$. Suppose that $y$ is $(b,c)$-inverse of $a$, then $yay=y$, $yR=bR$ and $Ry=Rc$ by Lemma~\ref{abcire}.
By $bR=uR$ and $Rc=Rv$, we have $yay=y$, $yR=uR$ and $Ry=Rv$, that is $y$ is $(u,v)$-inverse of $a$.
The opposite implication can be proved in a similar manner.
\end{proof}

\begin{corollary} \label{bcbcwa}
Let $a,d,c\in R$. If $a$ is $(b,c)$-invertible, $bR=ebR$ and $Rc=Rcf$, then $a$ is $(eb,cf)$-invertible.
In this case, we have $\bc{a}=\bcw{a}$.
\end{corollary}

If we let $u=v=d$ in Lemma~\ref{bcuva}, then by Lemma~\ref{alongd} and Lemma~\ref{alongo} and Lemma~\ref{abcirl},
we have that \cite[Theorem 2.14]{KCC} is a corollary of Lemma~\ref{bcuva}.
It is well-known that if $a\in R$ is invertible along $d$ if and only if $a$ is $(d,d)$-invertible.
Thus we have that \cite[Theorem 2.13]{KCC} is a corollary of Lemma~\ref{bcuva}.

By Lemma~\ref{alongd} and \cite[Theorem 2.13]{KCC}, we have the following remark.
Since $a\in R$ is invertible along $d$ if and only if $a$ is $(d,d)$-invertible,
a natural question is when an element invertible along $d$ is $(b,c)$-invertible,
the following remark answers this question.
\begin{remark}
Let $a,b,c,d\in R$ with $b,c$ are regular, $dR=bR$ and $d^{\circ}=c^{\circ}$.
Then $a$ is $(b,c)$-invertible if and only if $a$ is invertible along $d$.
In this case, $(b,c)$-inverse of $a$ coincides with the inverse along $d$ of $a$.
\end{remark}

The following theorem is a generalization of some well-known results of the $\{1,4\}$-inverse and the group inverse.

\begin{theorem} \label{inofbca}
Let $a,b,c\in R$. Then the following are equivalent:
\begin{itemize}
\item[{\rm (1)}] there exists $y\in R$ such that $aya=a$, $yay=y$ and $yR=bR$;
\item[{\rm (2)}] $a$ is regular, $aR=abR$ and $(ab)^{\circ}=b^{\circ}$;
\item[{\rm (3)}] $a$ is regular and $R=a^{\circ}\oplus bR$.
\end{itemize}
\end{theorem}
\begin{proof}
$(1)\Rightarrow(3)$. If there exists $y\in R$ such that $aya=a$, $yay=y$ and $yR=bR$,
then $a^{\circ} \cap bR=\{0\}$ by Lemma~\ref{pirdcca} and $yR=bR$.
Since $1=ya+(1-ya)\in yR+a^{\circ}=bR+a^{\circ}$ by $yR=bR$ and $aya=a$,
thus $R=a^{\circ}\oplus bR$.

$(3)\Rightarrow(2)$.
Suppose that $a$ is regular and $R=a^{\circ}\oplus bR$. Then $1=br+s$ for some $r\in R$ and $s\in a^{\circ}$.
Thus $a=a(br+s)=abr\in R$ by $as=0$, which gives that $aR=abR$.
Let $t\in (ab)^{\circ}$. Then $abt=0$ implies that $bt\in a^{\circ}$.
Since $bt\in bR$, thus $bt\in bR\cap a^{\circ}=\{0\}$ by $R=a^{\circ}\oplus bR$,
that is $bt=0$, thus $t\in b^{\circ}$. Therefore $(ab)^{\circ}=b^{\circ}$.

$(2)\Rightarrow(1)$.
Suppose that $a$ is regular, $aR=abR$ and $(ab)^{\circ}=b^{\circ}$. Then by Lemma~\ref{pirregulara} and $aR=abR$,
we have that $ab$ is regular.
By Lemma~\ref{annihilator} and $ab$ is regular, we have that $(ab)^{\circ}=b^{\circ}$ implies $Rb\subseteq Rab$.
Let $y=b(ab)^-$. We will check that $aya=a$, $yay=y$ and $yR=bR$.
The conditions $aR=abR$ and $ab$ is regular give that $a=ab(ab)^-a$, that is $a=aya$.
The conditions $Rb=Rab$ and $ab$ is regular give that $b=b(ab)^-ab$, then $yR=bR$ by $y=b(ab)^-$.
By $b=b(ab)^-ab$, we have $yay=b(ab)^-ab(ab)^-=b(ab)^-=y$.
\end{proof}

Let $a,y\in R$. By the proof of \cite[Theorem 3.1]{XCZ}, we have
\begin{equation} \label{12drpa}
aya=a,~yay=y,~ay^2=y,~ya^2=a~\Leftrightarrow~ay^2=y,~ya^2=a.
\end{equation}
\begin{equation} \label{12drpb}
aya=a,~yay=y,~y^2a=y,~a^2y=a~\Leftrightarrow~y^2a=y,~a^2y=a.
\end{equation}

If we take $y=a^\ast$ in Theorem~\ref{inofbca}, then there exists $y\in R$ such that $aya=a$, $yay=y$ and $yR=a^\ast R$
if and only if $a\in R^{\{1,4\}}$ by Lemma~\ref{134invese}.
Note that the condition $aR=aa^\ast R$ or $R=a^{\circ}\oplus a^\ast R$ implies that $a$ is regular by Lemma~\ref{134invese}.

If we take $y=a$ in Theorem~\ref{inofbca},
It is easy to check that there exists $y\in R$ such that $aya=a$, $yay=y$ and $yR=aR$ is equivalent to
$aya=a$, $yay=y$, $ay^2=y$ and $ya^2=a$. Thus, there exists $y\in R$ such that $ay^2=y$ and $ya^2=a$
if and only if $a\in R^\#$ by Lemma~\ref{pirgroupa} and $(\ref{12drpa})$.

In a similar manner, we have the following theorem, which a corresponding theorem of the Theorem~\ref{inofbca}.
The following theorem is a generalization of some well-known results of the $\{1,3\}$-inverse and the group inverse.
\begin{theorem} \label{inofbcdua}
Let $a,b,c\in R$. Then the following are equivalent:
\begin{itemize}
\item[{\rm (1)}] there exists $y\in R$ such that $aya=a$, $yay=y$ and $Ry=Rc$;
\item[{\rm (2)}] $a$ is regular, $Rca=Ra$ and ${}^{\circ}(ca)={}^{\circ}c$;
\item[{\rm (3)}] $a$ is regular and $R={}^{\circ}a\oplus Rc$.
\end{itemize}
\end{theorem}

By Theorem~\ref{inofbca} and Theorem~\ref{inofbcdua}, we have the following theorem.
In the following theorem, we answer the question when the $(b,c)$-inverse of $a$ is an inner inverse of $a$.
\begin{theorem} \label{inofbcbca}
Let $a,b,c\in R$. Then the following are equivalent:
\begin{itemize}
\item[{\rm (1)}] there exists $y\in R$ such that $aya=a$ and $y$ is the $(b,c)$-inverse of $a$;
\item[{\rm (2)}] $a$ is regular, $aR=abR$, $Rca=Ra$, $(ab)^{\circ}=b^{\circ}$ and ${}^{\circ}(ca)={}^{\circ}c$;
\item[{\rm (3)}] $a$ is regular, $R=a^{\circ}\oplus bR$ and $R={}^{\circ}a\oplus Rc$.
\end{itemize}
\end{theorem}

By Theorem~\ref{inofbcbca} and the properties of the Moore-Penrose inverse, group inverse and core inverse,
we have the following corollaries.

\begin{corollary} \label{corinoc}
Let $a\in R$. Then the following are equivalent:
\begin{itemize}
\item[{\rm (1)}] $a\in \core{R}$;
\item[{\rm (2)}] $aR=a^2R$, $Ra^\ast a=Ra$ and $(a^2)^{\circ}=a^{\circ}$;
\item[{\rm (3)}] \emph{\cite[Proposition 2.11]{XCZ}} $R=a^{\circ}\oplus aR$ and $R={}^{\circ}a\oplus Ra^\ast$;
\item[{\rm (4)}] \emph{\cite[Theorem 2.6]{XCZ}} $a\in R^\#\cap R^{\{1,3\}}$.
\end{itemize}
\end{corollary}
\begin{proof}
It is obvious by Lemma~\ref{134invese} and Lemma~\ref{pirgroupa}.
\end{proof}

\begin{corollary} \label{corinoa}
Let $a,x\in R$. Then the following are equivalent:
\begin{itemize}
\item[{\rm (1)}] \emph{\cite[Theorem 2.8]{RDDb}} $a^\dag=x$ if and only if $axa=a$, $xR=a^\ast R$ and $Rx=Ra^\ast$;
\item[{\rm (2)}] \emph{\cite[Theorem 2.7]{RDDb}} $a^\#=x$ if and only if $axa=a$, $xR=aR$ and $Rx=Ra$.
\end{itemize}
\end{corollary}

\centerline {\bf ACKNOWLEDGMENTS} The author is grateful to China Scholarship Council for giving him a purse for his further study in Universitat Polit\`{e}cnica de Val\`{e}ncia, Spain.

\end{document}